\documentclass[preprint,12pt]{elsarticle}

\usepackage[T2A]{fontenc}
\usepackage[utf8]{inputenc}
\usepackage[english]{babel}

\usepackage{amsmath}
\usepackage{amssymb}
\usepackage{amsthm}
\usepackage{bm}
\usepackage{epsfig}
\usepackage{pict2e}
\usepackage{tikz}

\graphicspath{{./images/}}

\newtheorem{theorem}{Theorem}[section]
\newtheorem{corollary}{Corollary}[theorem]
\newtheorem{lemma}[theorem]{Lemma}

\newtheorem{remark}{Remark}[theorem]
\theoremstyle{definition}
\newtheorem*{definition}{Definition}

\def\NN{{\mathbb N}}
\def\hats{{\sc Hats\ }}

\def\HG{{\mathcal G}}
\def\gplus#1{\mathop{+\!\!_{_{#1}}}}
\def\gtimes#1{\mathop{\times\!\!_{_{#1}}}}
\def\hgn{\text{\rm HG}}
\def\cnst#1{\star #1}

\begin{document}

\begin{frontmatter}

\journal{Discrete Mathematics}

\title{The Hats game. On maximum degree and diameter.}
\author[inst1,inst2]{Aleksei Latyshev}
\ead{aleksei.s.latyshev@gmail.com}
\address[inst1]{ITMO University, St. Petersburg, Russia}
\address[inst2]{Leonard Euler International Institute at St. Petersburg, St.Petersburg, Russia}

\author[inst3]{Konstantin Kokhas}
\ead{kpk@arbital.ru}
\address[inst3]{St. Petersburg State University, St. Petersburg, Russia}

\begin{abstract}
    We analyze the following version of the deterministic \hats game. We have a graph $G$, and a sage resides at each vertex of $G$. When the game starts, an adversary puts on the head of each sage a hat of a color arbitrarily chosen from a set of $k$ possible colors. Each sage can see the hat colors of his neighbors but not his own hat color. 
    All of sages are asked to guess their own hat colors simultaneously, according to a 
    predetermined guessing strategy and the hat colors they see, where no communication between them is allowed. The strategy is winning if it guarantees at least one correct individual guess for every assignment of colors. Given a graph $G$, its hat guessing number $\hgn(G)$ is the maximal number $k$ such that there exists a winning strategy. 
    
    We disprove the hypothesis that $\hgn(G) \le \Delta + 1$ and demonstrate that diameter of graph and $\hgn(G$) are independent. 
\end{abstract}

\end{frontmatter}

\section{Introduction}

The \hats game goes back to an old popular Olympiad problem. Its generalization to arbitrary graphs attracted the interests of mathematicians recently 
(see, e.g. %
\cite{bosek19_hat_chrom_number_graph, he20_hat_guess_books_windm,   alon2020hat}).

In this paper, we consider a variant of a hat guessing game. We have a graph $G$, and a sage resides at each vertex of $G$. There is some adversary who plays against the sages. When the game starts, the adversary puts on the head of each sage a hat of a color that he chooses from a set of $k$ possible colors. Each sage can see the hat colors of his neighbors but not his own hat color. 
All of sages are asked to guess their own hat colors simultaneously, according to their predetermined guessing strategy and the hat colors they see, where no communications between them is allowed. 
The sages act as a team, they can only discuss their strategy before the game begins.
The strategy is winning if it guarantees at least one correct individual guess for every assignment of colors. Given a graph $G$, its hat guessing number $\hgn(G)$ is the maximal number $k$ such that there exists a winning strategy.

The maximum number $k$, for which the sages can
guarantee the win, is called the \emph{hat guessing number} of graph $G$ and denoted
$\hgn(G)$. Computation of the hat guessing number for an arbitrary graph is a hard
problem. Currently it is solved only for few classes of graphs: for complete graphs,
trees (folklore), cycles~\cite{cycle_hats}, and pseudotrees~\cite{Kokhas2018}. Also, there are some results for ``books'' and ``windmills'' graphs, see~\cite{he20_hat_guess_books_windm}, \cite{kokhas_hats_2021}.
For bipartite, multipartite and $d$-degenerate graphs some estimations of hat guessing numbers are obtained by Alon et al.~\cite{alon2020hat}, Gadouleau and Georgiou \cite{Gadouleau2015}, He and Li \cite{he_Li_20_hat_guess_dgnr_graphs}.

M. Farnik \cite{Farnik2015} considered relation of the hat guessing number and the maximal degree of a graph. Using Lov\'asz Local Lemma he proved that  $\hgn(G)< e \Delta(G)$ for any graph $G$. 
He proposed hypothesis that the stronger inequality $\hgn(G)\le \Delta + 1$ holds.
This hypothesis is also presented in \cite{bosek19_hat_chrom_number_graph} and~\cite{he20_hat_guess_books_windm}. We disprove this hypothesis in section~4.

In our previous works~\cite{kokhas_cliques_I_2021}
and~\cite{kokhas_cliques_II_2021} (joint with V.\,Retinsky), we considered the
version of the hats game with variable number of hats (i.e., the number of possible colors 
is a function on graph vertices). This version of game is not only of its own
interest, but gives more flexible approach to the classic \hats game analysis.
We proved several theorems, which allow us to build new winning graphs by 
combining the graphs, for which their winning property is already proved.
We call these theorems constructors, they provide us a powerful machinery for building the sages strategies. This machinery was partially extended by Bla\v{z}ej et al. in \cite{blazej_bears_2021}, where the technique of constructors is combined with the independence polynomials approach.

In this paper, we continue to study the \hats game with variable number of colors and obtain some results about hat guessing numbers. 
Using ideas of  Bla\v{z}ej et al. \cite{blazej_bears_2021}, we keep track the positivity of independence polynomials that allows us to find the hat guessing numbers of some special graphs obtained by constructors. In this way we construct graphs which hat guessing number equals $(4/3)\Delta$, where $\Delta$ is the maximum degree. 
We give a complete answer to the question 5.5 from~\cite{he20_hat_guess_books_windm} and
demonstrate how one can build a graph $G$ with an arbitrary diameter and independently at the same moment an arbitrary hat guessing number. It is natural to demand here the graph $G$ to be minimal in the sense that hat guessing number of every proper subgraph is less than $\hgn(G)$.

\bigskip

The rest of the paper is organized as follows.
In section 2 we give necessary definitions and notations, recall several
theorem-constructors from~\cite{kokhas_cliques_I_2021,  kokhas_cliques_II_2021, blazej_bears_2021}. 
In section 3 in terms of independence polynomials we define maximal games and prove theorem that 
connects maximal games with constructors. Construction of graphs $G$ for which $\hgn(G)=\Delta+N$ or $\hgn(G)=(4/3) \Delta$ are given in section 4.
In section 5 we define a clique extension of graph and prove several technical results concerning clique extensions and their independence polynomials. We demonstrate how this technique can be applied for estimating of the smallest root of the independence polynomial. In section 6 we prove that diameter and hat guessing number are independent parameters of a graph.

\section{Definitions, notations, constructors}

\subsection{Definitions and notations}

We use the following notations.

$G = \langle {V, E} \rangle$ is a visibility graph, i.\,e.\ graph with the sages
in its vertices. We identify the sages with the graph vertices.

$h\colon V\to \NN$ is a \emph{``hatness'' function}, which means the number of
different hat colors, which a sage can get. For a sage $A\in V$, we will call the
number $h(A)$ the \emph{hatness} of sage $A$. We may assume
that the hat color of sage $A$ is the number from the set $[h(A)]=\{0, 1, 2, \dots, h(A)-1\}$.

$g\colon V\to \NN$ is a \emph{``guessing'' function} that determines the number of guesses each sage is allowed to make.

We often consider functions on $V$ as vectors and denote them in boldface: $h=(h(v))_{v\in V}=\bm{h}=\bm{h}_V$. For $W\subset V$ the restriction of vector $\bm{x}=\bm{x}_V$ to $W$ is denoted $\bm{x}_W$.
A function that is equal to a constant $m$ we will denote $\cnst{m}$.

\begin{definition}
  A hat guessing game or \hats for short is a pair $\HG=\langle {G, h} \rangle$, where $G$ is a visibility graph, and $h$ is a hatness function. So, the sages are located in the vertices of the visibility graph $G$ and participate in the \emph{test}. During the test every sage $v$ gets a hat of one of $h(v)$ colors. The sages do not communicate and try to guess colors of their own hats. And if at least one of their guesses is correct, the sages \emph{win}, or the game is \emph{winning}. In this case, we say that the graph is winning too, keeping in mind that this property depends on the hatness function. The games where the sages have no winning strategy we call \emph{losing}.

In \cite{blazej_bears_2021} Bla\v{z}ej  et al.\ consider the following generalized hat guessing game in which multiple guesses are allowed. A \emph{generalized} or \emph{non-uniform} hat guessing game is a triple $\HG=\langle {G, h, g} \rangle$, where $G$ is a visibility graph, $h$ is a hatness function and $g$ a guessing function.  In this game every sage $A$
makes simultaneously $g(A)$ guesses during the test.
The sages win if for at lest one of sages the color of his hat coincides with one of his guesses. 

It is clear that \hats game  $\langle {G, h} \rangle$ is the same as the generalized game $\langle {G, h, \star1 } \rangle$. The classic \hats game where the  hatness function has constant value $m$ is denoted  $\langle {G, \cnst{m}} \rangle$.
\end{definition}

\subsection{Constructors}

We call constructors the theorems, which allow us to build new winning graphs by combining the graphs, for which their winning property is already proved. Here are several constructors from the papers~\cite{kokhas_cliques_I_2021,  kokhas_cliques_II_2021, blazej_bears_2021}.

\begin{figure}[h]

\begin{minipage}[t]{0.49\linewidth}
\begin{center}
\setlength{\unitlength}{.6mm}\footnotesize
\begin{picture}(40,38)(10,5)
\put(10,0){\circle*{2}}\put(20,0){\circle*{2}}
\put(10,5){\circle*{2}}\put(20,10){\circle*{2}}
\put(10,20){\circle*{2}}\put(20,20){\circle*{2}}
\put(10,0){\line(1,0){10}}\put(10,0){\line(0,1){20}}
\put(20,0){\line(0,1){10}}\put(20,0){\line(-2,1){10}}
\put(10,20){\line(1,0){10}}
\put(20,10){\line(0,1){10}}\put(20,10){\line(-1,1){10}}
\put(25,8){$S$}\put(26,21){$v$}
\put(-1,8){$G_1$}
\put(20,37){$G_2$}
\put(16,16){\circle{18}}
\put(25,25){
\put(0,0){\circle*{2}}\put(5,15){\circle*{2}}
\put(10,0){\circle*{2}}\put(20,20){\circle*{2}}\put(22,10){\circle*{2}}
\polyline(22,10)(10,0)(0,0)(5,15)(20,20)(22,10)(5,15)(10,0)(20,20)
}
\end{picture}
\begin{picture}(33,38)(-5,5)
\put(-10,10){\vector(1,0){7}}
\put(10,0){\circle*{2}}\put(20,0){\circle*{2}}
\put(10,5){\circle*{2}}\put(20,10){\circle*{2}}
\put(10,20){\circle*{2}}\put(20,20){\circle*{2}}
\put(10,0){\line(1,0){10}}\put(10,0){\line(0,1){20}}
\put(20,0){\line(0,1){10}}\put(20,0){\line(-2,1){10}}
\put(10,20){\line(1,0){10}}
\put(20,10){\line(0,1){10}}\put(20,10){\line(-1,1){10}}
\put(25,8){$S$}
\put(16,16){\circle{18}}
\put(25,25){
\put(5,15){\circle*{2}}
\put(10,0){\circle*{2}}\put(20,20){\circle*{2}}\put(22,10){\circle*{2}}
\polyline(22,10)(10,0)(-5,-5)(5,15)(20,20)(22,10)(5,15)(10,0)(20,20)
\polyline(10,0)(-15,-5)(5,15)(-5,-15)(10,0)
}
\end{picture}
\end{center}
  \caption{Game $G_1 \gplus{S,v} G_2$}
  \label{fig:multiplication}
 \end{minipage} 
\hfil
\begin{minipage}[t]{0.49\linewidth}  
\begin{center}
\setlength{\unitlength}{.6mm}\footnotesize
\begin{picture}(35,38)(5,5)
\put(10,0){\circle*{2}}\put(20,0){\circle*{2}}
\put(10,5){\circle*{2}}\put(20,10){\circle*{2}}
\put(10,20){\circle*{2}}\put(20,20){\circle*{2}}
\put(10,0){\line(1,0){10}}\put(10,0){\line(0,1){20}}
\put(20,0){\line(0,1){10}}\put(20,0){\line(-2,1){10}}
\put(10,20){\line(1,0){10}}
\put(20,10){\line(0,1){10}}\put(20,10){\line(-1,1){10}}

\put(16.3,16){$v$}\put(28.7,16){$v$}
\put(-1,8){$G_1$}
\put(30,39){$G_2$}
\put(28,20){
\put(0,0){\circle*{2}}\put(5,15){\circle*{2}}
\put(15,0){\circle*{2}}\put(20,20){\circle*{2}}\put(20,10){\circle*{2}}
\polyline(20,10)(15,0)(0,0)(5,15)(20,20)(20,10)(5,15)(15,0)(20,20)
}
\end{picture}
\begin{picture}(43,38)(-10,5)
\put(-10,8){\vector(1,0){7}}
\put(10,0){\circle*{2}}\put(20,0){\circle*{2}}
\put(10,5){\circle*{2}}\put(20,10){\circle*{2}}
\put(10,20){\circle*{2}}\put(20,20){\circle*{2}}
\put(10,0){\line(1,0){10}}\put(10,0){\line(0,1){20}}
\put(20,0){\line(0,1){10}}\put(20,0){\line(-2,1){10}}
\put(10,20){\line(1,0){10}}
\put(20,10){\line(0,1){10}}\put(20,10){\line(-1,1){10}}
\put(16.3,16){$v$}
\put(20,20){
\put(0,0){\circle*{2}}\put(5,15){\circle*{2}}
\put(15,0){\circle*{2}}\put(20,20){\circle*{2}}\put(20,10){\circle*{2}}
\polyline(20,10)(15,0)(0,0)(5,15)(20,20)(20,10)(5,15)(15,0)(20,20)
}
\end{picture}
\end{center}
  \caption{Game $G_1 \times_{v} G_2$}
  \label{fig:product}
\end{minipage}  
\end{figure}

\begin{definition}
  Let $G_1=(V_1, E_1)$, $G_2=(V_2, E_2)$ be two graphs, 
  $S\subseteq G_1$ be a clique, 
  and $v\in V_2$. Set $G=(V,E)$ to be the clique join of graphs $G_1$ and $G_2$ with respect to $S$ and $v$  (fig.~\ref{fig:multiplication}). We say that graph $G$ is \emph{a sum of graphs $G_1$, $G_2$ with respect to $S$ and $v$} and denote it by  $G=G_1\gplus{S,v}G_2$. We say that vector $\mathbf x =(x_v)_{v\in V}$ is a \emph{gluing} of vectors $\mathbf x_1 =(x_{1v})_{v\in V_1}$ and $\mathbf x_2 =(x_{2v})_{v\in V_2}$ and denote it by $\mathbf x =\mathbf x_1 \gplus{S,v} \mathbf x_2$, if 
$$
x_v=\begin{cases} 
x_{1u}x_{2v} & u\in S,\\   
x_{1u} & u\in V_1\setminus S,\\
x_{2u} & u\in V_2\setminus\{v\}.\end{cases}
$$
   Let $\HG_1 =\left\langle G_1, h_1 \right\rangle$, $\HG_2 =\left\langle G_2, h_2\right\rangle$ be two games. \emph{A sum} of games $\HG_1$, $\HG_2$ with respect to $S$ and $v$ or $(S,v)$-sum for short is a game $\HG =\langle {G_1 \gplus{S,v} G_2,  h_1 \gplus{S,v} h_2} \rangle$ (fig.~\ref{fig:multiplication}).
  We denote the sum by $\HG_1\gplus{S,v}\HG_2$. The sum of generalized hat guessing games $\HG_1 =\left\langle G_1, h_1, g_1 \right\rangle$, $\HG_2 =\left\langle G_2, h_2, g_2\right\rangle$
  is defined similarly: $\HG_1\gplus{S,v}\HG_2=\langle {G_1 \gplus{S,v} G_2,  h_1 \gplus{S,v} h_2, g_1 \gplus{S,v} g_2} \rangle$.
\end{definition}

\begin{theorem}[on sum of games]  \label{thm:sum-of games}
  Let $\HG_1 =\left\langle G_1, h_1, g_1 \right\rangle$, $\HG_2 =\left\langle G_2, h_2, g_2\right\rangle$ be two winning games, 
  $S\subseteq G_1$ be a clique, 
  and $v\in V_2$. Then the game $\HG = \HG_1 \gplus{S,v} \HG_2$ is also winning.
\end{theorem}

This theorem was proven in \cite[Lemma 7]{blazej_bears_2021}. 
For partial case where $S$ is a~single vertex, a sum of \hats games was considered in \cite{kokhas_cliques_I_2021} where it is called a product of games. 
We will use this operation often so let us described it in more details.

Let $\HG_1 =\left\langle G_1, h_1 \right\rangle$, $\HG_2 =\left\langle G_2, h_2\right\rangle$ be two games, and let one vertex in $G_1$ and one vertex in $G_2$ be marked $A$. Let $S=\{A\}\subset V_1$, $v=A\in V_2$. \emph{A~product of games $\HG_1$, $\HG_2$ with respect to vertex $A$} is just a $(\{A\},v)$-sum of $\HG_1$ and $\HG_2$.  We will denote product of games by more elegant notation $\HG = \HG_1 \gtimes{A} \HG_2$ from~ \cite{kokhas_cliques_I_2021}. For vectors and functions on $V$ we will use the sign $\gplus{A}$ instead of $\gplus{\{A\},A}$. 

The following theorem on game products was proven in \cite[Theorem 3.1]{kokhas_cliques_I_2021}, according to current views it is just a corollary of the above theorem.

\begin{corollary}[theorem on game products]  \label{thm:multiplication}
  Let $\HG_1 = \langle {G_1, h_1} \rangle$ and $\HG_2 = \langle {G_2, h_2}
  \rangle$ be two games such that $V(G_1)\cap V(G_2)=\{A\}$. If the sages win in
  games $\HG_1$ and $\HG_2$, then they win also in game  $\HG = \HG_1\gtimes{A}
  \HG_2$.
\end{corollary}

We need one more constructor from  \cite{kokhas_cliques_I_2021}.
By the \emph{substitution of graph~$G_1$ to graph~$G_2$ on the place of vertex~$v$} we call the graph $G_1\cup (G_2\setminus\{v\})$ with adding of all edges, that connect each vertex of $G_1$ with each neighbor of~$v$, see  fig.~\ref{fig:subs}. %

\begin{figure}[h]
\footnotesize
\setlength{\unitlength}{300bp}%
\begin{center}
  \begin{picture}(1,0.1290902)%
    \put(0,0){\includegraphics[width=\unitlength,page=1]{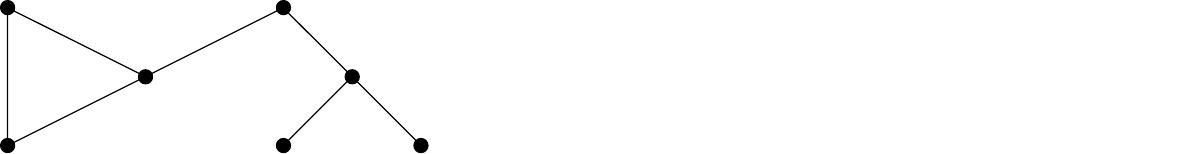}}%
    \put(0.11107517,0.03557355){\makebox(0,0)[lb]{\smash{$v$}}}%
    \put(0.48261333,0.07033853){\makebox(0,0)[lb]{\smash{$\longrightarrow$}}}%
    \put(0.45943632,0.02398464){\makebox(0,0)[lb]{\smash{$v:=$}}}%
    \put(0,0){\includegraphics[width=\unitlength,page=2]{subs.pdf}}%
  \end{picture}
\end{center}
  \caption{A substitution.}\label{fig:subs}
\end{figure}

\begin{corollary}\label{thm:substitution} %
  Let the sages win in games $\HG_1=\langle {G_1, h_1} \rangle$ and $\HG_2=\langle
  {G_2, h_2} \rangle$, where $G_1$ is a complete graph. Let $v\in V(G_2)$ be an arbitrary vertex and $G$ be the graph of substitution $G_1$ on place $v$. Then the game $\HG=\langle{G, h} \rangle$ is
  winning, where
  \[
    h(u) = \begin{cases}h_2(u)&u\in G_2\\ h_1(u)\cdot h_2(v)&u\in G_1\end{cases}
  \]
\end{corollary}

If $G_1$ is a complete graph the substitution is a partial case of $(S,v)$-sum of games for $S=V(G_1)$.  The general theorem on substitutions claims that the substitution game is winning if $\HG_1$ and $\HG_2$ are arbitrary winning games, see \cite[Theorem 3.2]{kokhas_cliques_I_2021}.

\begin{theorem}\label{thm:sum-lose}\cite[Theorem 4.1]{kokhas_cliques_II_2021}
  Let $G_1$ and $G_2$ be graphs such that $V(G_1)\cap  V(G_2)=\{A\}$,  $G=  G_1+_AG_2$.  And let games $\HG_1=\langle G_1, h_1\rangle$ and $\HG_2=\langle G_2, h_2\rangle $ be losing, $h_1(A)\geq h_2(A)=2$. Then game $\HG=\langle G_1+_AG_2, h\rangle$ is losing, where
  $$
  h(x) =
  \begin{cases}
    h_1(x),&x\in V(G_1)\\
    h_2(x),&x\in V(G_2)\setminus A.
  \end{cases}
  $$
\end{theorem}

\begin{theorem}\label{thm:2leaf-attach-lose}
  Let $\HG = \langle {G, h} \rangle$ be a loosing game, $B$ be an arbitrary
  vertex of graph~$G$. Consider graph $G'=\left\langle V', E' \right\rangle$,
  obtained by attaching new pendant vertex~$A$ to graph~$G$: $V'=V\cup \{A\}$,
  $E'=E\cup \{AB\}$. Then the sages loose in game $\langle G',h'\rangle$, where
  $h(A)=2$,  $h'(B)=2h(B)-1$ and $h'(u)=h(u)$ for other vertices $u\in V$.
\end{theorem}

To apply these constructors-theorems we need ``bricks'', i.\,e.\ examples of
winning (or losing) graphs. The following theorem provides big set of examples.

\begin{theorem}\label{thm:clique-win}\cite[theorem 2.1]{kokhas_cliques_I_2021}
The \hats game $\langle {K_n, h} \rangle$ is winning if and only if 
\begin{equation}
    \sum_{v\in V(G)}\frac1{h(v)}\geq 1.
    \label{eq:clique-win}
  \end{equation}
\end{theorem}

For generalized hat guessing games $\langle {K_n, h, g} \rangle$ the winning condition is equivalent to the inequality $ \sum\limits_{v\in V(G)}\frac{g(v)}{h(v)}\geq 1$ \cite[theorem 5]{blazej_bears_2021}.

We say that a game on complete graph $\langle {K_n, h, g} \rangle$ is \emph{precise} if 
$\sum\limits_{v\in V(G)}\frac{g(v)}{h(v)}= 1$.

\section{Independence polynomials and maximal games}

Let $G=\langle {V, E} \rangle$ be a graph. For the set of variables $\mathbf{x}=(x_v)_{v\in V}$ we define \emph{independence polynomial} of $G$ as
$$
P_{G}(\mathbf{x})=\sum_{\substack{I\subseteq V \\ I \ \text{independent set}}} \prod_{v\in I} x_v
$$
(the empty set is assumed to be independent and $\prod_{v\in\varnothing} x_v=1$).
Following \cite{blazej_bears_2021}, we consider the \emph{signed independence polynomial}
$$
Z_{G}(\mathbf{x})=P_{G}(-\mathbf{x}).
$$
The \emph{monovariate signed independence polynomial} $U_{G}(x)$ is obtained by plugging $-x$ for each variable $x_v$ of $P_{G}$.

For $v\in V$ we denote $N^+(v)$ the closed neighborhood of $v$, i.e. the set consisting of $v$ and all its neighbors. For any clique $K\subset G$ the independence polynomials $P_G$, $Z_G$ satisfy the recurrence relations 
\begin{equation}\label{eqn:indpolrec}
\begin{aligned}
P_G(\mathbf x) &= P_{G\setminus K}({\mathbf x})+ \sum_{u\in K} x_{u}P_{G\setminus N^+(u)}({\mathbf x}),\\ 
Z_G(\mathbf x) &= Z_{G\setminus K}({\mathbf x})- \sum_{u\in K} x_{u}Z_{G\setminus N^+(u)}({\mathbf x}).
\end{aligned}
\end{equation}

If $\mathbf x=(x_v)_{v\in V}$, $\mathbf y=(y_v)_{v\in V}$ are vectors, we write $\mathbf x\le\mathbf y$ if $x_v\le y_v $ for each coordinate $v$.
We write $\mathbf x\lneqq\mathbf y$ if $\mathbf x\le\mathbf y$ and $\mathbf x\ne\mathbf y$.
In other words, $x_v\le y_v $ for each coordinate $v$ and for at least coordinate the sign of inequality is strict.

Bla\v{z}ej at al. in \cite{blazej_bears_2021} develop theory of fractional hat guessing number.
For any visibility graph $G$ they consider \emph{fractional hat chromatic number}  $\hat\mu(G)$ defined as
$\hat\mu(G)=\sup\{\frac hg\mid \langle G, h, g \rangle \text{ is a winning game} \} $.
It is clear that $\hgn(G)\le \hat\mu(G)$. They prove the following facts for  generalized hat guessing games, see \cite[Proposition 9, 10 and Corollary 12]{blazej_bears_2021}.:

\begin{itemize}
\item
$\langle G, h, g \rangle$ is losing whenever $Z_G(\mathbf r)>0$, where $\mathbf r=(g_v/h_v)_{v\in V}$. %

\item
If there is a \emph{perfect} winning strategy for the hat guessing game $\langle G, h, g\rangle$, i.e. in every hat arrangement, no two sages that guess correctly are on adjacent vertices, then $Z_G(\mathbf r) = 0$ for $\mathbf r = (g_v /h_v )_{v\in V}$ and $Z_G (\mathbf w) \ge 0$ for every $0\le \mathbf w \le \mathbf r$. %

\item
For chordal graphs $G$  $\hat\mu(G)=1/r $, where $r$ is the smallest positive root of $U_G(x)$. %
\end{itemize}

Our next aim is to modify the third statement in order to control positivity of  $Z_G (\mathbf w)$.

An observation.  Let $G=(V,E)$ be a complete graph and the game $\langle G, h, g \rangle$ be precise. Then $Z_G(\mathbf x)=1-\sum\limits_{v\in V} x_v$ and the following statements hold.

1) $Z_G(\mathbf r) = 0$, where $\mathbf r=({g_v}/{h_v})_{v\in V}$.

2)  $Z_G(\mathbf x) >0$, where $\mathbf0 \le \mathbf x\lneqq \mathbf r$.

We say that the (generalized) game on an arbitrary graph $G$ is \emph{maximal} if it satisfies conditions 1) and 2). The maximal game can be winning or losing, but if we increase the hatness function (or decrease the number of guesses) in the maximal game, the game becomes loosing due to positivity of $Z_G$. Moreover, statement 2) guarantees that $Z_{G\setminus S}(\mathbf x) >0$ for any $S\subset V$, because the polynomial $Z_{G\setminus S}$ can be obtained from $Z_G$ by substitutions $x_v=0$ 
for each $v\in S$. %

\begin{theorem}\label{thm:product-maximality}
Let $\HG_1=\langle G_1, h_1, g_1 \rangle$ and $\HG_2=\langle G_2, h_2, g_2 \rangle$ be two maximal games, $S\subseteq G_1$ be a clique,
and $v\in V_2$. Then the game $\HG = \HG_1 \gplus{S,v} \HG_2$ is also maximal.
\end{theorem}

\begin{proof}
The sum $\HG_1\gplus{S,v}\HG_2$ is the game $\langle G, h, g \rangle$, where
$$
G=G_1\gplus{S,v} G_2, \qquad h= h_1 \gplus{S,v} h_2, \qquad g= g_1 \gplus{S,v} g_2.
$$

Let $\mathbf r=({g_v}/{h_v})_{v\in V}$, $\mathbf r=\mathbf r_1 \gplus{S,v} \mathbf r_2$, where $\mathbf r_1$ and $\mathbf r_2$ are similar vectors for $G_1$ and $G_2$. We will check that $Z_G(\mathbf r) = 0$. By maximality condition $Z_{G_1}({\mathbf r}_1) = 0$, $Z_{G_2}({\mathbf r}_2) = 0$. Applying the recurrence \eqref{eqn:indpolrec} (we take $K=S$ for~$G_1$ and $K=\{v\}$ for $G_2$) we obtain
\begin{equation}\label{eq:base-ZeroIndPol2}
\begin{aligned}
Z_{G_1}({\mathbf r}_1)&=Z_{G_1\setminus S}({\mathbf r}_1)- \sum_{u\in S} r_{1u}Z_{G_1\setminus N^+(u)}({\mathbf r_1})=0, 
\\
Z_{G_2}({\mathbf r}_2)&=Z_{G_2\setminus\{v\}}({\mathbf r}_2)- r_{\!2v}Z_{G_2\setminus N^+(v)}({\mathbf r}_2)=0.
\end{aligned}
\end{equation}
We write the recurrence \eqref{eqn:indpolrec} for $G$ taking $K=S$
\begin{align*}
Z_G(\mathbf r) &= Z_{G\setminus S}({\mathbf r})- \sum_{u\in S} r_{u}Z_{G\setminus N^+(u)}({\mathbf r}) 
\\&=
Z_{G_1\setminus S}({\mathbf r_1})Z_{G_2\setminus \{v\}}({\mathbf r_2}) - \sum_{u\in S} r_{1u}r_{2v}Z_{G_1\setminus N^+(u)}({\mathbf r_1})Z_{G_2\setminus N^+(v)}({\mathbf r_2}).
\end{align*}
Here the equality between the first and the second lines is based on the fact that the deletion of any vertex of $S$ splits graph $G$ into two components belonging to $G_1$ and $G_2$. Substituting \eqref{eq:base-ZeroIndPol2} in the last expression we obtain that $Z_G(\mathbf r)=0$.

Now we will prove that $Z_G(\mathbf x) >0$ for $\mathbf0\le \mathbf x\lneqq \mathbf r$. Let 
$\mathbf x=\mathbf x_1 \gplus{S,v} \mathbf x_2$, where $\mathbf0 \le \mathbf x_1\underset{(*)}{\le} \mathbf r_1$, $\mathbf0 \le \mathbf x_2\underset{(*)}{\le} \mathbf r_2$. 
Then for some vertex $v_0$ the sign of at least one of the inequalities ($*$) is strict.
Due to maximality of $\HG_1$ and $\HG_2$ the values $Z_{G_1\setminus S}({\mathbf x}_1)$, $Z_{G_1\setminus N^+(u)}({\mathbf x}_1)$, $Z_{G_2\setminus\{v\}}({\mathbf x}_2)$, $Z_{G_2\setminus N^+(v)}({\mathbf x}_2)$ are positive and instead of equalities \eqref{eq:base-ZeroIndPol2} we have the inequalities
$$
\begin{aligned}
Z_{G_1\setminus\{S\}}({\mathbf x}_1) &\ge \sum_{u\in S} x_{1u}Z_{G_1\setminus N^+(u)}({\mathbf x_1}), 
\\
Z_{G_2\setminus\{v\}}({\mathbf x}_2) &\ge x_{\!2v}Z_{G_2\setminus N^+(v)}({\mathbf x}_2).
\end{aligned}
$$
(for the graph $G_i$ that contains $v_0$ the sign of inequality is strict). Then
\begin{multline*}
Z_G(\mathbf x) = Z_{G_1\setminus S}({\mathbf x_1})Z_{G_2\setminus \{v\}}({\mathbf x_2}) 
- \\-
\sum_{u\in S} x_{1u}Z_{G_1\setminus N^+(u)}({\mathbf x_1})x_{2v}Z_{G_2\setminus N^+(v)}({\mathbf x_2})>0.
\end{multline*}
\end{proof}

\begin{corollary}\label{cor:hgn=mu}
Let the game $\HG=\langle G, \cnst{h} \rangle$ be obtained by a sequence of sum operations from a set of precise winning \hats games on complete graphs. Then $\hgn(G)=h$.
\end{corollary}

\begin{proof}
We use theorem~\ref{thm:product-maximality} for $g_1\equiv 1$, $g_2\equiv 1$, $g\equiv 1$. Then we conclude by induction based on theorem~\ref{thm:product-maximality} that game $\HG$ is maximal. $\HG$ is winning by theorem~\ref{thm:sum-of games} on sum of games and therefore $\hgn(G)\ge h$. It is clear that $G$ is chordal graph. It follows from maximality condition that $1/h$ is the smallest positive root of $U_G(x)$. Therefore $h=\hat\mu(G)$. It remains to apply general inequality $\hgn(G)\le \hat\mu(G)$.
\end{proof}

\begin{corollary}\label{cor:hgn=mu-gen}
Let the game $\HG=\langle G, h, g \rangle$ be obtained by a sequence of sum operations from a set of precise winning games on complete graphs. Let $h/g$ be a constant function, $h/g=h_0\in\mathbb Q$. Then $\hat\mu(G)=h_0$.
\end{corollary}

\begin{proof}
It is clear that $G$ is chordal graph.  In fact in the proof of the theorem~\ref{thm:product-maximality} it was checked that $1/h_0$ is the minimal positive root of $U_G$. Then $\hat\mu(G)=h_0$ by \cite[Corollary 12]{blazej_bears_2021}.
\end{proof}

We remark that for generalized hat guessing games a ratio $h(A)/g(A)$ can be interpreted in the spirit of \cite{blazej_bears_2021} as a fractional hatness of sage $A$. 
It can happen that the number $h_0$ in corollary \ref{cor:hgn=mu-gen} is integer.
Unfortunately, we can not claim in this case that $\hgn(G)=h_0$.
See remark \ref{rem:mu(chain)>hgn}.

\section{Hat guessing number and the maximal degree}

It is well known from folklore that $\hgn(G)\leq e \Delta(G)$ (Lov\'asz Local Lemma,
 see e.\,g. \cite[Theorem 2.4, Remark 2.5]{Farnik2015})), but
no examples of graphs for which $\Delta(G)+1<\hgn(G)$ are known. We provide such examples here.

We start from a nice concrete graph.

\begin{lemma}\label{lem:deg6-hg8}
For graph $G$ depicted in fig.~\ref{fig:deg6-hg8} (on the right) we have $\Delta(G)=6$ and $\hgn(G)=8$.
\end{lemma}

\begin{proof}
  The games on small complete graphs in parentheses in fig~\ref{fig:deg6-hg8} on the left, are winning precise games by theorem \ref{thm:clique-win} (the value of hatness function is written near each vertex). We combine these graphs by theorem \ref{thm:multiplication} on game product, and after that once again multiply three copies of the obtained graph.
  
  We obtain graph $G$ for which the game $\langle G, \cnst 8 \rangle$ is winning.
By corollary \ref{cor:hgn=mu} \ $\hgn(G)= 8$.
\end{proof}

\begin{figure}[t]
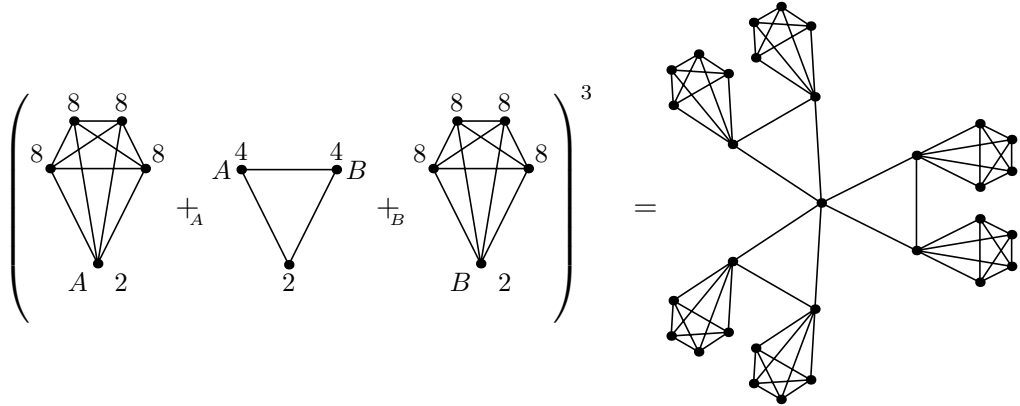

$$    
\left(\lower1cm\hbox{\epsfig{scale=1.2,file=planar41.mps}}
\gplus{A}
\lower1cm\hbox{\epsfig{scale=1.2,file=planar5.mps}}
\gplus{B}
\lower1cm\hbox{\epsfig{scale=1.2,file=planar42.mps}}\right)^3\quad=
\lower2.5cm\hbox{\epsfig{scale=1.2,file=planar3.mps}}
$$
    \caption{Graph $G$ for which $\Delta(G)=6$ and $\hgn(G)=8$.}
    \label{fig:deg6-hg8}
\end{figure}

\begin{lemma}

  \begin{enumerate}
  \item For any positive integer $k$ there exists a graph $G$ such that $\hgn(G)=\Delta(G) + k$.
  \item There exists a sequence of graphs $G_n$ such that $\Delta(G_n)\to +\infty$
    and $\lim\limits_{n \to +\infty} \hgn(G_n)/\Delta(G_n) = 8/7$.
  \end{enumerate}
\end{lemma}

\begin{proof}
We take the graph from lemma \ref{lem:deg6-hg8} and substitute the games $\langle {K_n, \cnst{n}} \rangle$ in place of each vertex. The values of hatness function on the obtained graph $G_n$ are equal to $8n$, and $\Delta(G_n)= 7n-1$. By corollary \ref{thm:substitution} graph $G_n$ is winning. Since the substitution of a complete graph is a partial case of the sum operation, $\hgn(G_n)=8n$ by corollary~\ref{cor:hgn=mu}. 
\end{proof}

\begin{theorem}
  There exists a sequence of graphs $G_n$ such that $\Delta(G_n)\to +\infty$ and
  $\hgn(G_n)/\Delta(G_n)= 4/3$.
\end{theorem}

\begin{proof}
  Let us fix some integer $n\ge 3$. Denote by $T_k$ for $k=1$, \dots, $n - 1$ the precise game on graph $K_{2^{n-k} + 1}$ pictured in fig.~\ref{fig:Tk}.

\begin{figure}[h]%
\begin{center}
\epsfig{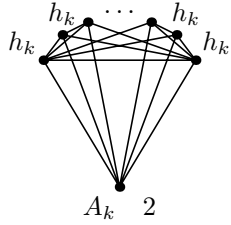}\hfill
\end{center}
\caption{Precise game $T_k$ on $K_{2^{n-k}+1}$.  The hatness of the $2^{n-k}$ top vertices is $h_k=2^{n-k+1}$, the bottom 
  vertex $A_k$ has hatness 2.}%
\label{fig:Tk}%
\end{figure}

Let us construct graph $G_n$. We start from graph $\tilde{G}_n$, which
is an essential part of the construction. The vertices of graph $\tilde{G}_n$ are
located on $n$ levels, enumerated from $0$ to $n-1$ in downward direction. To describe the
graph consider a vertex $A_k$, $0 \leq k\leq n-1$, on the $k$-th level (see fig.~\ref{fig:4/3} for $n=5$).

\begin{figure}[h]%
\begin{center}
\epsfig{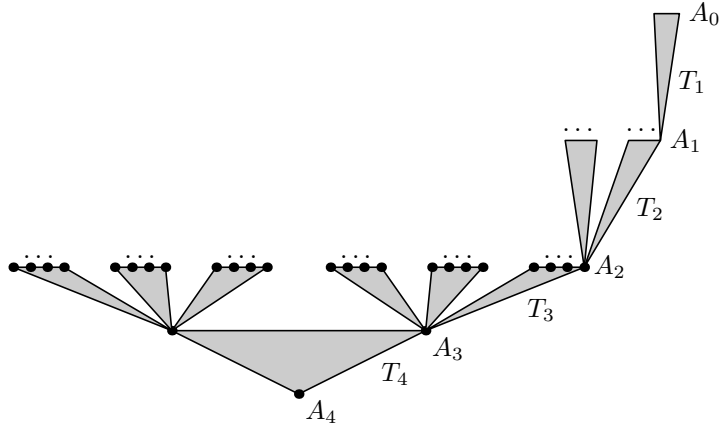}\hfill
\end{center}
\caption{Scary example $\tilde{G}_n$, $n=5$}%
\label{fig:4/3}%
\end{figure}

The vertex $A_{n-1}$ is on the lowest level, $\deg A_{n-1}=2$, $h(A_{n-1})=2$.

Let us describe the vertices $A_k$, $0\leq k\leq n-2$. The vertex
$A_k$ belongs to clique $T_{k+1}$, which bottom vertex lies on the $(k+1)$-th
level, and the top vertices on the $k$-th level are adjacent to $A_k$. Thus
there are $2^{n-k-1}$ edges from the vertex $A_k$ to vertices of $T_{k+1}$. Moreover, the vertex $A_k$ is also a bottom vertex in $k$ cliques of kind $T_k$. The other vertices of these cliques are located on the $(k-1)$-th level, so there are $k\cdot 2^{n-k}$ edges from $A_k$ to the vertices of these cliques. In total $\deg A_k = k\cdot 2^{n-k}+2^{n-k-1}$. Note that 
$$
\max\deg A_k = \max\{k\cdot 2^{n-k}+2^{n-k-1}, k= 0, 1, \dots, n-2\}
$$
is reached for $k=1$ and equals $2^{n-1}+2^{n-2}=3/4 \cdot 2^n$.

Graph $\tilde{G}_n$ is constructed by theorem \ref{thm:product-maximality} on game product as a product of the cliques $T_k$ mentioned above. In these products the vertex $A_k$ plays the role of the bottom vertex of $T_{k}$ $k$ times (with hatness $2$), and it plays the role of the top vertex of $T_{k+1}$ once (with hatness $2^{n-k}$).  Thus $h(A_k)=2^{n-k}\cdot2^k=2^n$.

Now let graph $G_n$ be a product of $n$ copies of graph $\tilde{G}_n$ by vertex $A_{n-1}$. 
Then $\deg A_{n-1} =2n$,  $h(A_{n-1})=2^n$, the degrees and hatnesses of the other
vertices remain unchanged. Since the hatness function is constant on the obtained graph, by corollary \ref{cor:hgn=mu} $\hgn(G_n)=2^n$ and the following equality holds.
$$
\frac{\hgn(G_n)}{\Delta(G_n)}=\frac{2^n}{\frac34 2^n}  = \frac43.
$$
\end{proof}

\section{Independence polynomials of clique extensions}
\label{sec:IPofCE}

\paragraph{Clique extension of the first kind}
Let $G=\langle {V, E} \rangle$ be a graph, where $V=\{v_1, \dots, v_n\}$, and
let $K_{a_1}$, \ldots, $K_{a_n}$ be a set of complete graphs with one marked
vertex $v_i$ in each of them. \emph{The clique extension of the first kind} of the graph G is the graph
$$
G\gplus{v_1}K_{a_1}\gplus{v_2}K_{a_2}\gplus{v_3}\ldots\gplus{v_n}K_{a_n},
$$
denote it by $G_n=G(a_1, a_2, \dots, a_n)$, see fig.~\ref{fig:clq-ext-example}.

\begin{figure}
\begin{center}
\epsfig{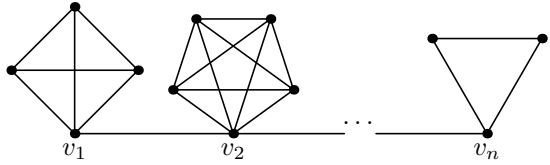}
\end{center}
    \caption{Graph $G(4, 5, \dots, 3)$ is a clique extension of the first kind of path  $P_n$.}
    \label{fig:clq-ext-example}
\end{figure}

The univariate alternating independence polynomial $U_{G_n}(x)$ of graph $G_n$ is obtained from $P_{G_n}(\mathbf{x})$ by identifying all the variables $x_v$ with new variable~$x$. Observe  that any anticlique in the graph $G_n$ contains at most one vertex from every clique $K_{a_j}$. Let us ``partially'' identify the variables $x_v$: for each $i$ let $x_i$ be a new variable corresponding to clique $K_{a_i}$, put $x_v=x_i$ in $P_{G_n}(\mathbf{x})$  for all vertices $v$ of clique $K_{a_i}$. Denote the obtained \emph{reduced} independence polynomial by $\widetilde{P}_{G_n}(\mathbf{x})$, $\mathbf{x}=(x_1, \dots, x_n)$.

In the following theorem we give an analogue of the recurrence relation \eqref{eqn:indpolrec} for reduced independence polynomials.

\begin{theorem}\label{thm:cl-ext-1}

  Let $d=\deg v_n$ and $v_{n-1}, \dots, v_{n-d}$ be the neighbors of the vertex
  $v_n$ in graph $G$. Then the independence polynomial $\widetilde{P}_{G_n}$
  satisfies the equation
  \begin{multline}\label{eqn:PGn-recurr}
    \widetilde{P}_{G_n}(\mathbf x)=
    \widetilde{P}_{G_{n-1}}(\mathbf x) + x_{n}(a_n-1)\cdot \widetilde{P}_{G_{n-1}}(\mathbf x)
\\
+\widetilde{P}_{G_{n-d-1}}(\mathbf x)\cdot x_{n}\cdot \prod_{j=1}^d x_{n-j}\bigl(a_{n-j}-1 +\tfrac 1{x_{n-j}}\bigr).
  \end{multline}
\end{theorem}

\begin{proof}
The left hand side counts independent sets of $k$ vertices in graph $G_n$. Let us construct by induction on $n$ a one to one correspondence between these sets and monomials of degree $k$ in the right hand side of the equation. Let $\Omega$ be an independent set of size $k$. The following three cases are possible.

  \begin{enumerate}
  \item $\Omega \cap V(K_{a_n})=\varnothing$. Then we can put into correspondence to $\Omega$ one of $k$-th degree terms in the first summand $\widetilde{P}_{G_{n-1}}(\mathbf x)$ of~\eqref{eqn:PGn-recurr}.

  \item $\Omega \cap V(K_{a_n})=\{v\}$, $v\ne v_n$ (recall that vertex $v$ corresponds to variable $x_n$ in the reduced independence polynomial). We can choose vertex $v$ in $a_n-1$ ways, and the other vertices of~$\Omega$ are chosen in $G_{n-1}$, this choice is counted by $(k - 1)$-th degree terms of $\widetilde{P}_{G_{n-1}}(\mathbf{x})$. In this way we obtain the second summand of formula~\eqref{eqn:PGn-recurr}.

  \item $v_n\in \Omega$. Then vertices $v_{n-j}$,
    $1\leq j\leq d$, do not belong to $\Omega$, but $\Omega$ can contain any other vertices from cliques $K_{a_{n-j}}$. Expanding the product $x_v\prod\limits_{j=1}^d x_{n-j}(a_{n-j}-1 +\tfrac 1{x_{n-j}})$ we see that every $m$-th degree term, $1\leq m\leq d$, determines a way of choosing $m$ elements to the set~$\Omega$. The choice of the other $k-m$ elements is encoded in  $\widetilde{P}_{G_{n-d-1}}(\mathbf{x})$. This calculation is hidden in the third summand of formula~\eqref{eqn:PGn-recurr}.
  \end{enumerate}
 \end{proof}

Rewrite the relation~\eqref{eqn:PGn-recurr} by uniting first two summands. 
\begin{multline}\label{eqn:PGn-recurr-short}
  \widetilde{P}_{G_n}(\mathbf x) = x_{n}   \bigl(a_n-1+\tfrac1{x_n}\bigr)\cdot \widetilde{P}_{G_{n-1}}(\mathbf x)
  +\\+
  \widetilde{P}_{G_{n-d-1}}(\mathbf x)\cdot x_{n}\cdot \prod_{j=1}^d x_{n-j}\bigl(a_{n-j}-1 +\tfrac 1{x_{n-j}}\bigr).
\end{multline}

Let for a moment $G(a_1, a_2, \dots, a_n)$ be a non degenerate clique extension in a sense that all $a_i>1$.
Let $f_n(a_1, a_2, \dots, a_n)$ be the leading coefficient of polynomial
$\widetilde{P}_{G_n}$, i.\,e.\ the number of ways to choose an independent set with
$n$ vertices in the graph $G(a_1, a_2, \dots, a_n)$. Denote 
$$
f_n=f_n(a_1, a_2, \dots, a_n)
$$
for short. It is obvious from combinatorial point of view and immediately follows from~\eqref{eqn:PGn-recurr-short} that sequence $f_n$ satisfies the recurrence relation
\begin{equation}\label{eqn:fG-recurrence}
  f_n = (a_n - 1)f_{n-1} + f_{n-d-1}\cdot\prod_{j=1}^d (a_{n-j} - 1).
\end{equation}
Hence $f_n$ is a polynomial of $n$-th degree in variables $a_1$, $a_2$, \dots, $a_n$, which has degree 1 with respect to each variable $a_i$. Thus, we may assign an arbitrary real values for the variables $a_i$ in this polynomial.

\begin{corollary}\label{cor-UGn}
  \begin{gather}\label{eqn:cor-UGn-upd}
    \widetilde{P}_{G_n}(\mathbf x) = x_1x_2\dots x_n  \cdot f_n\bigl(a_1+\tfrac1{x_1}, a_2+\tfrac 1{x_2}, \dots, a_n+\tfrac 1{x_n}\bigr),\\
    \label{eqn:cor-UGn}
    U_{G_n}(x) =(-x)^n \cdot f_n\bigl(a_1-\tfrac 1x, a_2-\tfrac 1x, \dots, a_n-\tfrac 1x\bigr).
  \end{gather}
\end{corollary}

\begin{proof}
  For the clique extensions of the one-vertex graph $G$ or for graph $G$
  consisting of one edge, formula~\eqref{eqn:cor-UGn-upd} can be checked
  directly. If we substitute $a_j := a_j + 1/x_j$ for all $j$,
  multiply both sides by $x_1x_2\dots x_n$ and then denote 
  $x_1x_2\dots x_kf_k$ by $\widetilde{P}_{G_k}$, we obtain
  formula~\eqref{eqn:PGn-recurr-short}. Thus, the left and the right hand sides
  of~\eqref{eqn:cor-UGn-upd} obey the same initial conditions and the
  same recurrence relations.

  After the substitution $x_j := -x$ into~\eqref{eqn:cor-UGn-upd}, we get~\eqref{eqn:cor-UGn}.
\end{proof} 

\begin{theorem}\label{thm:indpol}
  Let $G$ be an arbitrary graph, $G(a_1, a_2, \dots, a_n)$ $(a_i > 1)$ be its clique
  extension, and $f_n=f_n(a_1, a_2, \dots, a_n)$ be its polynomial function. Then
  \begin{equation}\label{eqn:indpol}
    U_{G}(x) =(-x)^n \cdot f_n\bigl(1-\tfrac 1x, 1-\tfrac 1x, \dots, 1-\tfrac 1x\bigr).
  \end{equation}
\end{theorem}

\begin{proof}
  Formulae~\eqref{eqn:PGn-recurr} and~\eqref{eqn:fG-recurrence} have important property: when we subsequently apply these formulas for fixed graph $G$ in order to derive the explicit formula for polynomials   $\widetilde{P}_{G_n}$ or $f_n$, the number of steps in this procedure  
does not depend on values $a_j$. Moreover, the formulae remain correct if some (or even all)  parameters $a_j$ are set to $1$. For example, for $a_n=1$ the second summand of~\eqref{eqn:PGn-recurr} is equal to $0$, that means, that for $a_n=1$ it is impossible to include into an independent set the vertex of clique $K_{a_n}$ which is not equal $v_n$. 
Thus the statement of corollary~\ref{cor-UGn} is valid if we substitute  $a_1=a_2=\ldots=a_n=1$ in~\eqref{eqn:cor-UGn}.
\end{proof}

\paragraph{Clique extension of the second kind}

Let $H=\langle {V, E} \rangle$ be a graph in which $V=\{v_1, \dots, v_n\}$, and
$K_{a_1}$, \ldots, $K_{a_n}$ be a set of complete graphs, where $a_j\geq \deg v_j$ and
$\deg v_j$ vertices are marked in each $K_{a_j}$. Let us consider a
graph $\bigcup\limits_{j=1}^n K_{a_j}$ and add several bridges between its components: 
for every edge $u_iu_j\in E(H)$ draw an edge connecting a marked vertex of $K_{a_i}$ with a marked vertex of $K_{a_j}$ (each marked vertex must be a vertex of one bridge only). The obtained graph is called \emph{a clique extension of the second kind} of graph $H$ (see fig.~\ref{fig:clique-chain-ex3}). Denote it by
$$
H_n=H(a_1, a_2, \dots, a_n).
$$

\begin{figure}[h]    
  \centering
  \begin{tikzpicture}[scale=.7]\footnotesize
    \foreach \shift/\size in {0/1, 3/2, 6/3, 12/{n-1}, 15/n} {
      \begin{scope}[shift={(\shift, 0)}]
        \draw (-1, 0) -- (0, 1) -- (1, 0) -- (0, -1) -- cycle;
        \draw node {$K_{a_{\size}}$};
      \end{scope}
    }
    \draw (9, 0) node {$\cdots$};
    \foreach \shift in {1, 4, 7, 10, 13} {
      \begin{scope}[shift={(\shift, 0)}]
        \draw (0, 0) -- (1, 0);
      \end{scope}
    }
  \end{tikzpicture}
\caption{Graph $H_n$ is a clique extension of the second kind of $H=P_n$. Each rhombus denotes a clique, which size is written in the center. Bridges are not incident.}
\label{fig:clique-chain-ex3}
\end{figure}
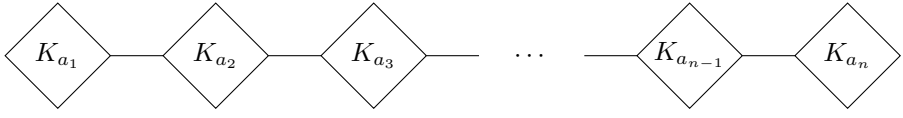

Let $H_n=H(a_1, a_2, \dots, a_n)$ be a clique extension of the second kind of graph $H$. \emph{Reduced} independence polynomial  $P_{H_n}(\mathbf{x})$, $\mathbf{x}=(x_1, \dots, x_n)$ is a polynomial obtained by plugging $x_j$ for each variable $v_v$,  $v\in K_{a_j}$, $j=1, \dots, n$.
As for extensions of the first kind, denote by $f_n=f_n(a_1, a_2, \dots, a_n)$ the leading coefficient of $\widetilde{P}_{H_n}$, i.\,e.\ the number of ways to choose an independent set with $n$ vertices in graph $H_n$.

\begin{theorem} Let $d=\deg v_n$ and let $v_{n-1}, \dots, v_{n-d}$ be the neighbors of vertex $v_n$ in graph $H$. For every $j$, $1\leq j\leq d$, denote by $H_{n-1}^{j}$ the graph $H(a_1, a_2, \dots, a_{n-j}-1, \dots, a_{n-1})$ which is obtained from $H_n$ by removing of clique $K_{a_n}$ and the bridge between $K_{a_n}$ and $K_{a_{n-j}}$ (including its endpoint in $K_{a_{n-j}}$). Then

1) $\widetilde{P}_{H_n}(\mathbf x)=
x_{n}\bigl(a_n-d+\tfrac1{x_n}\bigr)\cdot \widetilde{P}_{H_{n-1}}(\mathbf x)+
x_{n} \sum\limits_{j=1}^d \widetilde{P}_{H_{n-1}^{j}}(\mathbf x)$;

2) $U_{H_n}(x) =(-x)^n \cdot f_n\bigl(a_1-\tfrac 1x, a_2-\tfrac 1x, \dots, a_n-\tfrac 1x\bigr).$
\end{theorem}

The proof is similar to the proofs of theorem \ref{thm:cl-ext-1} and \ref{thm:indpol}.

\paragraph{Example 1}
Consider the clique extension of the first kind of path $G=P_n$: 
$$
G_n=G(k+1, k+1, \dots, k+1),
$$
Let
\begin{equation}\label{eqn:An-def}
A_n=A_n(k)=f_n(k+1,k+1,\dots,k+1), \qquad A_0=1.
\end{equation}
The recurrent relation \eqref{eqn:fG-recurrence} has a form
\begin{equation}\label{eqn:A-recurrence}
A_n=k(A_{n-1}+A_{n-2}), \qquad n\ge 2.
\end{equation}
Then $A_n$ are polynomials of $k$.

\begin{lemma}\label{thm:roots-An}
For every $n$ all real roots of polynomial $A_n(x)$ belong to $[-4,0]$.
\end{lemma}

\begin{proof}
It is clear that $A_n$ has no positive roots. Let real $k_0<-4$ be a root of $A_n$. 
Solving the linear recurrence \eqref{eqn:A-recurrence} we obtain

$$
A_n(k_0)=\frac{1+t_1}{k_0t_1^{n+1}(t_1-t_2)} + \frac{1+t_2}{k_0t_2^{n+1}(t_2-t_1)}, 
$$%
where $t_1$ and $t_2$ are the roots of its characteristic polynomial:
$$
t_1=-\frac12+\frac12\sqrt{1+\frac4{k_0}} \quad\text{and}\quad
t_2=-\frac12-\frac12\sqrt{1+\frac4{k_0}}.
$$
Equating this expression to 0 we obtain
$$
\frac{1+t_1}{1+t_2} = \biggl(\frac{t_1}{t_2}\biggr)^{n+1}.
$$
But it is impossible because the left hand side is greater than 1, and the right hand side is less than 1.
\end{proof}

\paragraph{Example 2} 
Consider one more clique extension of the first kind for $G=P_n$
\begin{equation}\label{eqn:primer2}
G_n=G(k+3, \underbrace{k+1, k+1,  \dots, k+1,}_{n-2} k+3).
\end{equation}
Let us find monovariate signed independence polynomial $U_{G_n}$ for this graph.

Let for $n\geq 2$
\begin{align*}
L_n&=L_n(k)=f_n(k+3, \underbrace{k+1, k+1,\dots, k+1,}_{n-2} k+3),
\\
B_n&=f_n(\,\underbrace{k+1,k+1,\dots, k+1,}_{n-1} k+3)=f_n(k+3,\underbrace{k+1,k+1,\dots, k+1}_{n-1} \,).
\end{align*}
The following recurrence relations for sequences $B_n$, $L_n$ and $A_n$ (from Example 1, see~\eqref{eqn:An-def}) hold for $n\geq 3$ similarly to \eqref{eqn:A-recurrence} 
$$
B_n=(k+2)A_{n-1}+kA_{n-2}, \qquad L_n=(k+2)B_{n-1}+kB_{n-2}.
$$
Applying \eqref{eqn:A-recurrence} for simplification we obtain for $n\geq 4$
\begin{align*}
L_n&=(k+2)\Bigl( (k+2)A_{n-2}+kA_{n-3} \Bigr) +k\Bigl( (k+2)A_{n-3}+kA_{n-4} \Bigr)
=\\&=
(k+2)\Bigl( A_{n-1}+2A_{n-2} \Bigr) +k\Bigl( A_{n-2}+2A_{n-3} \Bigr)
=\\&=
A_n+4A_{n-1}+4A_{n-2}
=\\&=
A_n\cdot\Bigl(1+\frac4k\Bigr).
\end{align*}
Though we see $k$ in the denominator, the right hand side is a polynomial on $k$ (one can prove that polynomial $A_n$ is divisible by $k^{\lfloor\frac n2\rfloor}$). Observe also that for $n=2$ and $n=3$ the formula is true.

By corollary \ref{cor-UGn} we conclude that
\begin{equation}\label{eqn:L-ind-pol}
U_{G_n}(x)=(-x)^n  f_n\bigl(k+1-\tfrac 1x, k+1-\tfrac 1x, \dots, k+1-\tfrac 1x\bigr)\cdot
\Bigl(1+\frac4{k-\frac 1x}\Bigr).
\end{equation}

\begin{theorem}\label{thm:minroot-Gn}
All real roots of polynomial $U_{G_n}(x)$ are positive and the minimal positive root equals $1 / (k+4)$.
\end{theorem}

\begin{proof}
The number $x=1 / (k+4)$ is the root of the last multiplier in \eqref{eqn:L-ind-pol}.
The other real roots are the roots of the first multiplier $f_n$ and by lemma~\ref{thm:roots-An}
they satisfy the inequality  $-4\leq k - 1 / x\leq 0$. It is clear that $x>0$ here and then 
$x\geq 1 / (k+4)$.
\end{proof}

\paragraph{Examlpe 3} Consider the clique extension of the second kind for graph $H=P_n$ (fig.~\ref{fig:clique-chain-ex3})
$$
H_n=H(k+1, \underbrace{k, k,  \dots, k,}_{n-2}  k+1).
$$

Let
\begin{align*}
E_n&=E_n(k)=f_n(k+1, \underbrace{k, k,\dots, k,}_{n-2} k+1),
\\
\Phi_n&=\Phi_n(k)=f_n(\underbrace{k, k,\dots, k}_{n}).
\end{align*}
Polynomials $\Phi_n(k)$ satisfy the recurrence
$$
\Phi_n=k \Phi_{n-1}-\Phi_{n-2}.
$$
Now one can check similarly to theorem \ref{thm:roots-An} that for every $n$ the real roots of $\Phi_n(k)$ belong to  $[-2; 2]$. 

Since $\Phi_n$ and $E_n$ satisfy the relation $E_n=(k+2)\Phi_{n-1}$, we obtain the following theorem.

\begin{theorem}\label{thm:minroot-for-Hn}
All real roots of polynomial $U_{H_n}(x)$ are positive and the minimal positive root equals $1 / (k+2)$.
\end{theorem}

\section{Hat guessing number and diameter}

In this section we prove that diameter and hat guessing number are independent parameters of a graph. 

Let $k>1$, $\ell$ and $n$ be positive integers. Consider graphs
$H_n^\ell$, $\widetilde H_n^\ell$ and $H_n^{\ell-}$ depicted in fig.~\ref{fig:clique-chains}.
$H_n^\ell$ and $\widetilde H_n^\ell$ are clique extensions of the second kind of path $P_n$ in terms of section~\ref{sec:IPofCE}. 
Each rhombus in the figure is a clique, which size is written in its center. The cliques are connected by bridges without common vertices.  One edge in the clique $K_{\ell-2}^{-}$ is removed.

 \begin{figure}[tb]    
  \centering\footnotesize
  \begin{tikzpicture}[scale=.7]
    \foreach \shift/\size in {0/1, 3/2, 6/2, 12/2, 15/1} {
      \begin{scope}[shift={(\shift, 0)}]
        \draw (-1, 0) -- (0, 1) -- (1, 0) -- (0, -1) -- cycle;
        \draw node {$K_{\ell - \size}$};
      \end{scope}
    }
    \draw (9, 0) node {$\cdots$};
    \foreach \shift in {1, 4, 7, 10, 13} {
      \begin{scope}[shift={(\shift, 0)}]
        \draw (0, 0) -- (1, 0);
      \end{scope}
    }
  \end{tikzpicture}

Graph $H_{n}^\ell$ is a clique extension of the second kind of path $P_n$.

\bigskip
\begin{tikzpicture}[scale=.7]
    \foreach \shift/\size in {0/1, 3/2, 6/2, 12/2, 15/2}{
      \begin{scope}[shift={(\shift, 0)}]
        \draw (-1, 0) -- (0, 1) -- (1, 0) -- (0, -1) -- cycle;
        \draw node {$K_{\ell - \size}$};
      \end{scope}
    }
    \draw (9, 0) node {$\cdots$};
    \foreach \shift in {1, 4, 7, 10, 13} {
      \begin{scope}[shift={(\shift, 0)}]
        \draw (0, 0) -- (1, 0);
      \end{scope}
    }
  \end{tikzpicture}

Graph $\widetilde H_n^\ell$. The size of the rightmost clique is decreased.

\bigskip
\begin{tikzpicture}[scale=.7]
    \foreach \shift/\size in {0/{{\ell-1}}, 6/{\ell-2}^{-}, 12/{{\ell-2}}, 15/{\ell-1} }{
      \begin{scope}[shift={(\shift, 0)}]
        \draw (-1, 0) -- (0, 1) -- (1, 0) -- (0, -1) -- cycle;
        \draw node {$K_\size$};
      \end{scope}
    }
    \draw (9, 0) node {$\cdots$};
    \draw (3, 0) node {$\cdots$};
    \foreach \shift in {1, 4, 7, 10, 13} {
      \begin{scope}[shift={(\shift, 0)}]
        \draw (0, 0) -- (1, 0);
      \end{scope}
    }
  \end{tikzpicture}

Graph $H_n^{\ell-}$. One edge in the clique $K_{\ell-2}^{-}$ is removed.

\medskip
\caption{Graphs $H_n^\ell$, $\widetilde H_n^\ell$ and $H_n^{\ell-}$. Each rhombus is a clique, which size is written in its center. The cliques are connected by bridges without common vertices.}
\label{fig:clique-chains}
\end{figure}
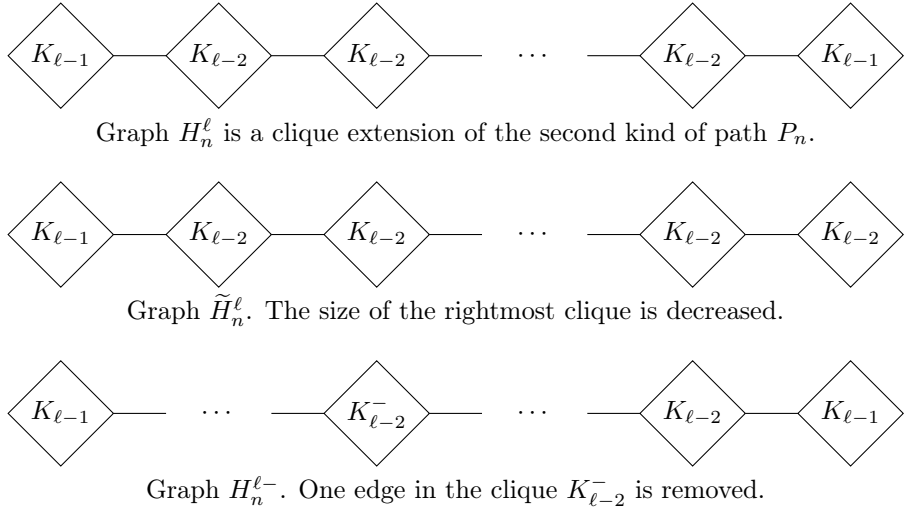

\begin{lemma} \label{lem:stegosaur-games}

$1)$ %
$\hgn(H_n^{2k})=2k$ and the game $\langle {H_n^{2k}, \cnst{2k}} \rangle$ is maximal winning.

$2)$ The game $\langle {H_n^{2k+1}, \cnst{(2k+1)}} \rangle$ is losing.

$3)$ The games $\langle {\widetilde H_n^{2k}, \cnst{2k}} \rangle$ and $\langle {\widetilde H_n^{2k-}, \cnst{2k}} \rangle$ are losing.
\end{lemma}

\begin{proof}
1) Consider the following three games.

\begin{equation}\label{eqn:bricks-for-cliquechain}
  \begin{tikzpicture}[scale=.7,baseline]\footnotesize
    \begin{scope}
      \draw (-1, 0) -- (0, 1) -- (1, 0) -- (0, -1) -- cycle;
      \draw (0, 0) node {$K_{2k - 1}$};
      \filldraw[black] (1, 0) circle (2pt) node[above right] {$k$};
    \end{scope}
    \begin{scope}[shift={(3.5, 0)}]
      \draw (0, 0) -- (1, 0);
      \filldraw[black] (0, 0) circle (2pt) node [above] {$2$};
      \filldraw[black] (1, 0) circle (2pt) node [above] {$2$};
    \end{scope}
    \begin{scope}[shift={(8, 0)}]
      \draw (-1, 0) -- (0, 1) -- (1, 0) -- (0, -1) -- cycle;
      \draw (0, 0) node {$K_{2k - 2}$};
      \filldraw[black] (1, 0) circle (2pt) node[above right] {$k$};
      \filldraw[black] (-1, 0) circle (2pt) node[above left] {$k$};
    \end{scope}
  \end{tikzpicture}
\end{equation}

\noindent
Hatnesses are written near bold vertices. Hatnesses of the other vertices are equal $2k$. These games are precise winning by theorem \ref{thm:clique-win}. It is clear that removing of any edge make them losing. Graph $H_n^{2k}$ is obtained from games~\eqref{eqn:bricks-for-cliquechain} by corollary~\ref{thm:multiplication} on game product in the following way.

\begin{equation}\label{eqn-pic-chainproduct}
  \begin{tikzpicture}[scale=.7]\footnotesize
    \begin{scope}
      \draw (-1, 0) -- (0, 1) -- (1, 0) -- (0, -1) -- cycle;
      \draw (0, 0) node {$K_{2k - 1}$};
      \filldraw[black] (1, 0) circle (2pt) node[above] {$\;\;k$};
    \end{scope}

    \foreach \shift in {1.6, 3.7, 6.9, 8.4, 11.6, 13.7} {
      \draw (\shift, 0) node {$\times$};
    }

    \foreach \shift in {2.1, 12.1} {
      \begin{scope}[shift={(\shift, 0)}]
        \draw (0, 0) -- (1, 0);
        \filldraw[black] (0, 0) circle (2pt) node [above] {$2$};
        \filldraw[black] (1, 0) circle (2pt) node [above] {$2$};
      \end{scope}
    }
    \foreach \shift in {5.3, 10.} {
      \begin{scope}[shift={(\shift, 0)}]
        \draw (-1, 0) -- (0, 1) -- (1, 0) -- (0, -1) -- cycle;
        \draw (0, 0) node {$K_{2k - 2}$};
        \filldraw[black] (1, 0) circle (2pt) node[above] {$\;\;k$};
        \filldraw[black] (-1, 0) circle (2pt) node[above] {$k\;\;$};
      \end{scope}
    }

    \begin{scope}[shift={(15.2, 0)}]
      \draw (-1, 0) -- (0, 1) -- (1, 0) -- (0, -1) -- cycle;
      \draw (0, 0) node {$K_{2k - 1}$};
      \filldraw[black] (-1, 0) circle (2pt) node[above] {$k\;\;$};
    \end{scope}

    \draw (7.6, 0) node {$\ldots$};
  \end{tikzpicture}
\end{equation}

\noindent
We glue pairs of vertices near each ``$\times$'' sign and multiply their hatnesses. By corollary~\ref{cor:hgn=mu} the obtained game $\langle {H_n^{2k}, \cnst{2k}} \rangle$ is maximal winning and  $\hgn(H_n^{2k})=2k$.

2) 
We start form graph $K_{2k}$ in which one vertex is labeled $A$. Define the hat function $h$ as
$h(A)=k+1$ and $h(v)=2k+1$ for $v\ne A$. The game $\langle K_{2k}, h \rangle$ is loosing by theorem \ref{thm:clique-win}, hence we can apply theorem~\ref{thm:2leaf-attach-lose} and obtain the following  loosing game.

\begin{center}  
\begin{tikzpicture}[scale=.7]\footnotesize
    \begin{scope}
      \draw (-1, 0) -- (0, 1) -- (1, 0) -- (0, -1) -- cycle;
      \draw node {$K_{2k}$};
      \filldraw[black] (-1, 0) circle (2pt) node [below] {$k+1\quad\;$};
      \draw (-1, 0) node [above] {$A\;\;$};
    \end{scope}
    \draw (2.5, 0) node {$\longrightarrow$};%
    \begin{scope}[shift = {(7, 0)}]
      \draw (-1, 0) -- (0, 1) -- (1, 0) -- (0, -1) -- cycle;
      \draw node {$K_{2k}$};
      \draw (-1, 0) node [above] {$A\;\;$};
      \draw (-3, 0) -- (-1, 0);
      \filldraw[black] (-1, 0) circle (2pt) node [below] {$2k\!+\!1\;\;\quad$};
      \filldraw[black] (-3, 0) circle (2pt) node [above] {2};
    \end{scope}
  \end{tikzpicture}
\end{center}

Now we multiply by theorem \ref{thm:sum-lose} several copies of loosing games that are equal or similar to the obtained game. For each vertex the value of the hatness is either 2 (then it is written near the vertex in the figure) or $2k+1$ (then it is omitted). 
The product operation of theorem \ref{thm:sum-lose} we denote by $\ltimes$.

\begin{center}  
\begin{tikzpicture}[scale=.7]\footnotesize
    \begin{scope}
      \draw (-1, 0) -- (0, 1) -- (1, 0) -- (0, -1) -- cycle;
      \draw node {$K_{2k}$};
      \filldraw[black] (1, 0) circle (2pt); %
    \end{scope}
    \draw (1.5, 0) node {$\ltimes$};%
    \begin{scope}[shift = {(4, 0)}]
      \draw (-1, 0) -- (0, 1) -- (1, 0) -- (0, -1) -- cycle;
      \draw node {$K_{2k - 1}$};
      \draw (-2, 0) -- (-1, 0);
      \filldraw[black] (-1, 0) circle (2pt); %
      \filldraw[black] (1, 0) circle (2pt); 
      \filldraw[black] (-2, 0) circle (2pt) node [above] {2};
    \end{scope}
    \draw (6.5, 0) node {$\ltimes\quad\dots\quad\ltimes$};
    \begin{scope}[shift = {(10, 0)}]
      \draw (-1, 0) -- (0, 1) -- (1, 0) -- (0, -1) -- cycle;
      \draw node {$K_{2k - 1}$};
      \draw (-2, 0) -- (-1, 0);
      \filldraw[black] (-1, 0) circle (2pt); %
      \filldraw[black] (1, 0) circle (2pt); 
      \filldraw[black] (-2, 0) circle (2pt) node [above] {2};
    \end{scope}
    \draw (11.5, 0) node {$\ltimes$};
    \begin{scope}[shift = {(14, 0)}]
      \draw (-1, 0) -- (0, 1) -- (1, 0) -- (0, -1) -- cycle;
      \draw node {$K_{2k}$};
      \draw (-2, 0) -- (-1, 0);
      \filldraw[black] (-1, 0) circle (2pt); %
      \filldraw[black] (1, 0) circle (2pt); 
      \filldraw[black] (-2, 0) circle (2pt) node [above] {2};
    \end{scope}
  \end{tikzpicture}
\end{center}

\noindent
By theorem  \ref{thm:sum-lose} the resulting game $\langle {H_n^{2k+1}, \cnst{(2k+1)}} \rangle$ is losing.

3) The game $\langle {\widetilde H_n^{2k}, \cnst{2k}} \rangle$ is loosing due to maximality of game $\langle {H_n^{2k}, \cnst{2k}} \rangle$. The game $\langle{\widetilde H_n^{-},\cnst{2k}}\rangle$ is loosing by constructions similar to the previous proof. We can check that the following two games are loosing (values of the hatness function that are not indicated in the figure equal $2k$).

\begin{center}\footnotesize
  \begin{tikzpicture}[scale=.7] %
    \begin{scope}
      \draw (-1, 0) -- (0, 1) -- (1, 0) -- (0, -1) -- cycle;
      \draw node {$K_{2k - 2}^{-}$};
      \filldraw[black] (-1, 0) circle (2pt) node [above left] {$k$};
      \filldraw[black] (1, 0) circle (2pt) node [above right] {$k$};
    \end{scope}
    \begin{scope}[shift = {(6, 0)}]
      \draw (-1, 0) -- (0, 1) -- (1, 0) -- (0, -1) -- cycle;
      \draw node {$K_{2k - 2}^{-}$};
      \draw (-3, 0) -- (-1, 0);
      \draw (3, 0) -- (1, 0);
      \filldraw[black] (-1, 0) circle (2pt) node [below] {$2k - 1\qquad$};
      \filldraw[black] (-3, 0) circle (2pt) node [above left] {2};
      \filldraw[black] (1, 0) circle (2pt) node [below] {$\qquad\; 2k - 1$};
      \filldraw[black] (3, 0) circle (2pt) node [above left] {2};
    \end{scope}
  \end{tikzpicture}
\end{center}

\noindent 
The latter game remains loosing if we increase the values of hatness function from $2k-1$ to $2k$. And then by theorem \ref{thm:sum-lose} we multiply the latter game from left and from right by suitable games $\langle {\widetilde H_i^{2k}, \cnst{2k}} \rangle$.
\end{proof} %

\begin{remark}\label{rem:mu(chain)>hgn}
The game $\langle {H_n^{2k+1}, \cnst{(2k+1)}} \rangle$ can be constructed as a product like \eqref{eqn-pic-chainproduct} of winning generalized hat guessing games (values of the hatness function that are not indicated in the figure equal $2k+1$):

\begin{center}
  \begin{tikzpicture}[scale=.7]\footnotesize
    \begin{scope}
      \draw (-1, 0) -- (0, 1) -- (1, 0) -- (0, -1) -- cycle;
      \draw (0, 0) node {$K_{2k}$};
      \filldraw[black] (1, 0) circle (2pt) node[above] {$\quad k+\tfrac12$};
    \end{scope}
    \foreach \shift in {1.6, 3.7, 6.9, 8.4, 11.6, 13.7} {
      \draw (\shift, 0) node {$\times$};
    }

    \foreach \shift in {2.1, 12.1} {
      \begin{scope}[shift={(\shift, 0)}]
        \draw (0, 0) -- (1, 0);
        \filldraw[black] (0, 0) circle (2pt) node [below] {$2$};
        \filldraw[black] (1, 0) circle (2pt) node [below] {$2$};
      \end{scope}
    }
    \foreach \shift in {5.3, 10.} {
      \begin{scope}[shift={(\shift, 0)}]
        \draw (-1, 0) -- (0, 1) -- (1, 0) -- (0, -1) -- cycle;
        \draw (0, 0) node {$K_{2k}$};
        \filldraw[black] (1, 0) circle (2pt) node[above] {$\quad k+\tfrac12$};
        \filldraw[black] (-1, 0) circle (2pt) node[above] {$k+\tfrac12\quad\; $};
      \end{scope}
    }

    \begin{scope}[shift={(15.2, 0)}]
      \draw (-1, 0) -- (0, 1) -- (1, 0) -- (0, -1) -- cycle;
      \draw (0, 0) node {$K_{2k}$};
      \filldraw[black] (-1, 0) circle (2pt) node[above] {$k+\tfrac12\quad\; $};
    \end{scope}

    \draw (7.6, 0) node {$\ldots$};
  \end{tikzpicture}
\end{center}
 
\noindent
Therefore by corollary \ref{cor:hgn=mu-gen} $\hat\mu(H_n^{2k+1})=2k+1>\hgn(H_n^{2k+1})$.
\end{remark}

Now we will show that diameter and hat guessing number are independent parameters of a graph. In the following theorem hat guessing number is even and diameter is odd but these arithmetical restrictions seem to be non crucial.

For any graph $G$ with $\hgn(G)\ge 3$ it is not difficult to construct graph $G'$ such that $G$ is a subgraph of $G'$, the diameter of $G'$ is as large as we wish and $\hgn(G)=\hgn(G')$. One of possible constructions is just subsequent attaching of new leafs, see \cite[theorem 3.8]{kokhas_cliques_I_2021} for details. In order to avoid such constructions we will consider \emph{minimal} graphs: let $h_0=\hgn(G)$, we say that graph $G$ is \emph{minimal} if $\hgn(G_0)<h_0$ for each subgraph $G_0$ of $G$ and the game $\langle G, \cnst{h_0}\rangle$ is maximal.

\begin{theorem}
For any odd $d$ and even $h_0 > 3$ there exists minimal graph $G$ with diameter $d$ 
and  $\hgn(G)=h_0$. 
\end{theorem}

\begin{proof}
 Let $G=H_n^{2k}$ (fig.~\ref{fig:clique-chains}), then $d = 2n + 1$, $h_0 = 2k$.  By lemma~\ref{lem:stegosaur-games}.1) and 3) the game $\langle G, \cnst{h_0}\rangle$ is maximal and $\hgn(H_n^{2k})=2k$.
 
 If we remove one edge from graph $G$, we obtain either $\tilde H_n^{2k}$ or $H_n^{2k-}$.
 By lemma~\ref{lem:stegosaur-games}.3) hat guessing number of both graphs are less than $2k$.
 Therefore $G$ is minimal.
\end{proof}

Fractional hat guessing number and diameter are also independent parameters of a graph.

\begin{theorem}
 For any integers $d\ge 1$ and $h > 3$ there exists graph~$G$ with diameter $d$ and $\hat{\mu}(G)=h$. 
\end{theorem}

\begin{proof} The examples of such graphs are $K_h$ (for $d=1$), $H_n^h$ (for $d=2n+1$, fig.~\ref{fig:clique-chains}),  $G_n$ (for $d=n+2$, $h=k+4$,  see \eqref{eqn:primer2}). By theorem \ref{thm:minroot-for-Hn} $\hat{\mu}(H_n^h)=h$, by theorem \ref{thm:minroot-Gn} $\hat{\mu}(G_n)=h=k+4$ (for $k=0$ the statement is still correct). Observe that there are no arithmetical restrictions in the latter example.
\end{proof}

In fact, the examples in the proof satisfy corollary \ref{cor:hgn=mu-gen}. In this way 
one can show that these graphs are minimal in a sense that fractional hat guessing number of each subgraph is less than $h$.

\section*{Acknowledgements}

The first author is supported by the Ministry of Science and Higher Education of the Russian Federation, agreement 075-15-2019-1620 date 08/11/2019.

\bibliographystyle{elsarticle-harv} 
\bibliography{main}

\end{document}